\documentclass[reqno,12pt,a4paper]{amsart}
\usepackage{pgfplots}
\usepackage[english]{babel}
 \pdfoutput=1
\usepackage{amsmath,amsthm,amssymb,amsfonts, mathdots}
\usepackage{bbm}
\usepackage{scrtime, cancel}
\usepackage{enumitem, color, comment}

\definecolor{pinegreen}{rgb}{0.0, 0.47, 0.44}
\definecolor{brilliantrose}{rgb}{1.0, 0.33, 0.64}
\usepackage{mathtools}
\usepackage{ifpdf}
\ifpdf
\usepackage[backref]{hyperref}
\else
\usepackage[hypertex]{hyperref}
\fi
\usepackage{pdfsync,verbatim}
\usepackage{color}
\mathtoolsset{showonlyrefs}

\numberwithin{equation}{section}   %%numera le equazioni sezione per sezione

\allowdisplaybreaks
\newtheorem{theorem}{Theorem}[section]
\newtheorem{lemma}[theorem]{Lemma}
\newtheorem{proposition}[theorem]{Proposition}
\newtheorem{corollary}[theorem]{Corollary}

\theoremstyle{definition}
\newtheorem{remark}[theorem]{Remark}

\usepackage{graphicx}

\makeatletter
\newcommand*\bigcdot{\mathpalette\bigcdot@{.5}}
\newcommand*\bigcdot@[2]{\mathbin{\vcenter{\hbox{\scalebox{#2}{$\m@th#1\bullet$}}}}}
\makeatother

\newcommand{\R}{\mathbb{R}}

\DeclareMathOperator{\tr}{tr}             % trace
\newcount\dotcnt\newdimen\deltay
\def\Ddot#1#2(#3,#4,#5,#6){\deltay=#6\setbox1=\hbox to0pt{\smash{\dotcnt=1
\kern#3\loop\raise\dotcnt\deltay\hbox to0pt{\hss#2}\kern#5\ifnum\dotcnt<#1
\advance\dotcnt 1\repeat}\hss}\setbox2=\vtop{\box1}\ht2=#4\box2}

\newcommand{\norm}[1]{\left\lVert#1\right\rVert}

\newcommand{\Red}{\color{red}}
\def\Blue{\color{blue}}

\pgfplotsset{compat = newest}

 \makeatletter
\def\author@andify{%
  \nxandlist {\unskip ,\penalty-1 \space\ignorespaces}%
    {\unskip {} \@@and~}%
    {\unskip \penalty-2 \space \@@and~}%
}
\makeatother

\title[]{
Weak type (1,1) jump inequalities \\ in a nonsymmetric Gaussian setting}

\author{Valentina Casarino}
\address{Valentina Casarino, DTG, Universit\`a degli Studi di Padova\\ Stradella san Nicola 3 \\I-36100 Vicenza \\ Italy}
\email{valentina.casarino@unipd.it}
\author{Paolo Ciatti}\address{Paolo Ciatti, Dipartimento di Matematica "Tullio Levi Civita", Universit\`a degli Studi di Padova\\Via Trieste, 63, 35131 Padova,  \\ Italy}
\email{paolo.ciatti@unipd.it}\author{Peter Sj\"ogren}\address{Peter Sj\"ogren, Mathematical Sciences,  University of Gothenburg and  Mathematical Sciences,
Chalmers University of Technology  \\ SE - 412 96 G\"oteborg, Sweden}\email{peters@chalmers.se}


\thanks{The first and second authors are members of the Gruppo Nazionale per l'Analisi Matematica, la Probabilità e le loro Applicazioni (GNAMPA)
of the Istituto Nazionale di Alta Matematica (INdAM). 
This research was carried out while the third author was visiting  the University of Padova,  Italy, and he is grateful for its hospitality. 
The third author also profited from a grant from the Royal Swedish Academy of Sciences and another from the Royal Society of Arts and Sciences in Gothenburg.}

\date{\today}

\begin{document}

\maketitle

\begin{abstract}
We prove that
 the jump quasi-seminorm of order $\varrho= 2$ for a general Ornstein--Uhlenbeck semigroup $\left(\mathcal H_t\right)_{t>0}$ in $\R^n$
 defines an operator of weak type $(1,1)$ with respect to the invariant  measure. This provides an example of a weak-type jump inequality for a  nonsymmetric semigroup
in a nondoubling measure space.

 Our result may be seen as  an endpoint refinement of the weak type $(1,1)$ inequality for  the $\varrho$-th order variation seminorm of  $\left(\mathcal H_t\right)_{t>0}$, recently proved by the authors when $\varrho>2$, and disproved for $\varrho=2$. 
 %The central part of the proof is based on a careful comparison between the Ornstein--Uhlenbeck semigroup and a convolution semigroup, for which jump inequalities are known.
%%By exploiting some hypercontractivity properties of $\left(\mathcal H_t\right)_{t>0}$,
%%%we also prove  that for large values of  $t$ 
%%the jump quasi-seminorms $J^p_\varrho (\mathcal H_t)$ are bounded on $L^p$
%%% for  $1<p<\infty$ and $\varrho\ge 1$. An analogous result for small values of $t$ is still unknown.

  \end{abstract}

\section{Introduction}\label{Introduction}
In this paper we study   jump and  variational 
bounds of   weak type $(1,1)$ for a  nonsymmetric Ornstein--Uhlenbeck semigroup,
 that is, for the  semigroup $(\mathcal H_t)_{t\ge 0}$ 
 generated by the operator
$$\mathcal L f(x)=\frac12\,\tr
\big(
Q\nabla^2
f\big)(x)+
\langle Bx, \nabla f(x) \rangle
\,,\qquad
f\in
\mathcal S (\R^n),\;\; x\in\R^n.
$$
Here $\nabla$ is the gradient and $\nabla^2$ the Hessian, 
and $n\ge 1$.
 The definition of $\mathcal L$ depends on two real $n\times n$ matrices, $Q$ and $B$; 
we assume that $Q$ is   symmetric and positive definite, and that all the eigenvalues of
$B$   have negative real parts.
In the following, $\R^n$ is  equipped
with a Gaussian measure denoted by $\gamma_\infty$, which 
is 
invariant  under the action of $\mathcal H_t$; see \eqref{def:inv_meas} 
for a precise definition of $\gamma_\infty$ in terms of $Q$ and $B$. This measure is locally but not globally doubling.
Under our assumptions on $Q$ and $B$, neither $\mathcal L$ nor $\mathcal H_t$ is symmetric in $L^2(\gamma_\infty)$ in general.

If $\mathcal I {\color{black}{\subseteq}} \R_+$ is an interval and $f \in L^1(\gamma_\infty)$,   the $\varrho$-th order variational seminorm of $\mathcal H_t f(x)$   on   $\mathcal I $ is defined by
           \begin{equation}\label{def:intro_rho2_OU}
  \|\mathcal H_t \,f(x)\|_{v(\varrho), \mathcal I} = \sup\, \left( \sum_{i=1}^N |\mathcal H_{t_i} f(x)- \mathcal H_{t_{i-1}}f(x)|^\varrho \right)^{1/\varrho},\qquad x\in \R^n,\quad
  1 \le \varrho < \infty,
\end{equation}
 the supremum being over all finite, increasing
 sequences $\left(t_i \right)_0^N$ of points in $\mathcal I$.
  We emphasize  that in this paper  the $v(\varrho)$ seminorm will be always taken in the variable~$t$.

 \hskip2pt
 
In a recent paper 
the authors proved that for $\varrho>2$  the variation operator mapping $f\in L^1(\gamma_\infty)$ to the function 
$$
x\mapsto \|\mathcal H_tf(x)\|_{v(\varrho),\R_+}
$$
is bounded 
from $L^1(\gamma_\infty)$ to $L^{1,\infty}(\gamma_\infty)$  (see \cite{CCS8}, and \cite{CCS6} for a different proof in the one-dimensional case).
 The   boundedness of  this variation operator on  $L^p(\gamma_\infty),\;1<p<\infty,$ 
 was already known 
 (see \cite[Corollary 4.5]{Le Merdy}, and also  \cite{Almeida}). 
            
These results break down for  $1\le \varrho\le 2$, as  proved in \cite[Section 8]{CCS8}.
%Since the variation seminorm  is decreasing with respect to  the parameter $\varrho\in [1, \infty)$, 
%no boundedness properties hold  for  $\varrho\in [1,2]$.\medskip

Our main result in this paper is a  jump inequality of weak type $(1,1)$
for $ \mathcal H_t$ at  $\varrho = 2$.
   We therefore need to introduce  jump quasi-seminorms,
   which we do only in the  Ornstein--Uhlenbeck framework.

If $\mathcal I {\color{black}{\subseteq}} \mathcal R_+$ is an interval,  $\lambda>0$ and $f \in L^1(\gamma_\infty)$, the $\lambda$-jump counting function for 
$\mathcal H_t\, f$ in $\mathcal I$ at   $x\in\R^n$
is defined as
\begin{multline}\label{def:jump}
 N_\lambda (\mathcal H_t f(x):\:t\in\mathcal I)=     %\right\}
 \sup\big\{
 J \in\mathbb N:\,
\text{  there exist points }\,
  t_1 < t_2 < \dots < t_J \text{ in $\mathcal I$ }   
 \\
\text{ such that }\,
|\mathcal H_{t_{i}} f(x) - \mathcal H_{t_{i-1}} f(x)| > \lambda
 \,\text{ for } \, i=2,\dots, J
 \big\}.
 \end{multline}
Then the jump quasi-seminorm  of order $\varrho$ of
$\mathcal H_t\, f$ in $\mathcal I$ is               %the supremum
\begin{equation}\label{jump_seminorm_OU}
 J_\varrho^p  (\mathcal H_t f:\,t\in\mathcal I)= %%%J_\varrho^p  (\mathcal H_t f(x,t)))=
 \sup_{\lambda>0}
 \| \lambda \, \big(N_\lambda \big(\mathcal H_t f(\cdot):  t\in\mathcal I \big)\big)^{1/\varrho}\|_{L^p(\gamma_\infty)},
\quad 1 \le p,\varrho < \infty.
\end{equation}
Notice that the quasi-seminorms $ J_\varrho^p(\cdot) $ are decreasing  in $\varrho$.

One can define 
weak jump quasi-seminorms  in a similar way, 
by replacing the Lebesgue space $L^p(\gamma_\infty)$   in  \eqref{jump_seminorm_OU}  by a   Lorentz space $L^{p,q}(\gamma_\infty)$
(see  \cite[formula (1.2)]{Mirek1});
we are mainly interested in  the case $(p,q)=(1,\infty)$, where 
\begin{equation}\label{jump_seminorm_weak}
 J_\varrho^{1,\infty}  (\mathcal H_t f:\,t\in\mathcal I)=
 \sup_{\lambda>0}
 \| \lambda \, \big(N_\lambda \big(\mathcal H_t f(\cdot):  t\in\mathcal I \big)\big)^{1/\varrho}\|_{L^{1,\infty}(\gamma_\infty)}.
\end{equation}

Since \begin{align*}
\lambda\big( N_\lambda (\mathcal H_t f(x):\,\,t\in\mathcal I)\big)^{1/\varrho}\le   \|\mathcal H_t \,f(x)\|_{v(\varrho), \mathcal I}\,, \qquad  x\in\R^n,
\end{align*}
for all $\varrho\ge 1$ and  $\lambda>0$  (see   %\cite[p. 6712, the first, not displayed, formula after (3)]{Jones2} or 
\cite[formula (2.3)]{Mirek1}),
 the  jump quasi-seminorms
$ J_\varrho^p $ are dominated by 
the $L^p$ norms of 
variational seminorms of order $\varrho$, that is,
\begin{equation}\label{domination}
 J_\varrho^p  (\mathcal H_t f:\,t\in\mathcal I) \le \| \,\|\mathcal H_t \,f(\cdot) \|_{v(\varrho), \mathcal I}\|_{L^p(\gamma_\infty)},\qquad \varrho \ge 1,\quad p\ge 1.
\end{equation}
It also follows that
\begin{align*}
 \big|\{x\in\R^n &: \lambda\left( N_\lambda (\mathcal H_t f(x):\,t\in\mathcal I)\right)^{1/\varrho}>\alpha \}\big| \\
&\le
\left|\{x\in\R^n:  \|\mathcal H_t f(x) \|_{v(\varrho), \mathcal I}>\alpha \}\right|,\qquad \alpha > 0,
\end{align*}
and thus
\begin{equation}\label{weak_domination_OU}
 J_\varrho^{1,\infty} (\mathcal H_t f:\,t\in\mathcal I) \le \| \,\|\mathcal H_t f(\cdot)\|_{v(\varrho), \mathcal I}\|_{L^{1,\infty}(\gamma_\infty)},\qquad \varrho\ge 1.
\end{equation}

Notice that combining \eqref{domination} with the results in \cite[Corollary 4.5]{Le Merdy} and \cite[p.\ 31]{Almeida} 
on the strong boundedness $L^p-L^p$ of the operator
$f \mapsto  \| \mathcal H_t f\|_{v(\varrho),\Bbb R_+}$, for $1<p<\infty$ and $\varrho>2$
one immediately arrives at the following result.
%As a consequence, in virtue of  \eqref{domination} the  jump quasi-seminorms 

\begin{proposition}%%%{\footnote{Dimostrato da Almeida e coll.? Non ho trovato nulla in letteratura.}}
Let $1<p<\infty$ and $\varrho>2$.
Then the  jump quasi-seminorms  of the  Ornstein--Uhlenbeck semigroup 
 $\left(\mathcal H_t\right)_{t>0}$ are bounded on $L^p(\gamma_\infty)$, that is,
\begin{equation*}%%\label{jump_seminorm_vv}
 J_\varrho^p  (\mathcal H_t f:\,t>0){\color{black}{\le  C}}\, \|f\|_{L^p(\gamma_\infty)},
\end{equation*}
with ${\color{black}{C=  C(p,n,Q,B)}}$.
\end{proposition}

 On the other hand, for $\varrho>2$ one has
\begin{equation}\label{domination3}
\|\,\|\mathcal H_t \,f\|_{v(\varrho), \mathcal I}\|_{L^{1,\infty}(\gamma_\infty)}
\lesssim_{\varrho}
 J_2^1 (\mathcal H_t \,f:\,t\in\mathcal I) \le \| \,\|\mathcal H_t \,f \|_{v(2), \mathcal I}\|_{L^1(\gamma_\infty)},
\end{equation}
with the first inequality failing for $\varrho=2$. 
See \cite[formula (1.4)]{Mirek2} or  \cite[Lemma 2.3]{Mirek1} for the first inequality; the second is \eqref{domination} with $p=1$.
Thus a  jump inequality of weak type $(1,1)$ for $ \mathcal H_t \,f$ at $\varrho=2$
can be considered an endpoint refinement of the variational estimate.

\medskip

%%In this paper, the focus is on jump inequalities for $1\le\varrho\le 2$.
 We recall that jump and variational 
 inequalities
say   that  the jump quasi-seminorms  or
 the $L^p$ norms of the variation seminorms of  
$(\mathcal H_t f)_{t\in\mathcal I}$
are
 bounded by the $L^p$ norm of $f$.
%%In the last year, they have been extensively studied in ergodic theory and harmonic analysis.
Since the pioneering works of L\'epingle, Gaposhkin and Bourgain \cite{Lepingle, Gaposhkin1, Gaposhkin2, Bourgain, Bourgain_etal},
this kind of  bounds has proved to be  very useful
to measure and control the fluctuations of families of linear operators, and it  often provides %%quite simple 
an answer to the problem of pointwise convergence.
 The range of   operators for which 
variational bounds have been proved in the last thirty years is  vaste
(see {\em e.g.} \cite{Jones,  Campbell,  Jones_Rosenblatt, Jones1-higher,Jones4, Harboure, Jones2,  Crescimbeni, Betancor1, Ma1, Ma2,  Liu,
  Mirek1, Mirek2, Mirek3, Seeger, Betancor2, Mirek4, Mirek5, Mirek6, Betancor3}).
But   only a few results
     seem to concern the case of  an  Ornstein--Uhlenbeck semigroup,
despite the prototypical nature of $(\mathcal H_t)_{t\ge 0}$ as a possibly nonsymmetric semigroup in a nondoubling measure space.

The standard case of a symmetric Ornstein-Uhlenbeck semigroup 
is given by  $Q=-B=I_n$, where $I_n$ is the $n$-dimensional identity matrix.
In that particular case, a variational approach was successfully addressed  in
 \cite{Harboure, Crescimbeni}, for $\varrho>2$.
  Recent contributions in the nonsymmetric Ornstein--Uhlenbeck context  are due to  Almeida, Betancor, Farina, Quijano and Rodriguez-Mesa
\cite{AlmeidaMedit,
Almeida}; they prove some interesting variational inequalities, under the condition $\varrho>2$.

The counterexample exhibited in  \cite[Section 8]{CCS8}, which is based on  earlier work due to Qian \cite{Qian}, says that 
 for $\varrho=2$
 the  variation seminorm of  $\left(\mathcal H_t \right)_{t>0}$
 is neither of strong nor of weak type $(p,p)$ for any $p \in [1,\infty)$, with respect to $\gamma_\infty$.
In the light of this counterexample and of \eqref{domination} and \eqref{weak_domination_OU},
it is reasonable to look for a weak type $(1,1)$ jump inequality at $\varrho=2$.
Our main result is the following.
\begin{theorem}\label{thm}
For   all $f\in L^1(\gamma_\infty)$ one has
\begin{equation}\label{ineq:weaktype} J_2^{1,\infty} (\mathcal H_t f:\,t>0){\color{black}{\le C}}\,\|f\|_{L^1(\gamma_\infty)},
 \end{equation}
{\color{black}{
with $C= C(n,Q,B)$.
}}
\end{theorem}
For $\varrho<2$ an analogous estimate %% weak type $(1,1)$ boundedness
 is not possible. Indeed, it has been proved in \cite[Lemma 2.3]{Mirek1} and in \cite[formula (2.20)]{Mirek4}, both applied with $r=2$ and $p=1$,  that
for all $1\le\varrho< 2$ one has
\begin{equation}\label{ineq:Mariusz}
\|\, \|\mathcal H_t \,f(x) \|_{v(2), \mathcal I}\|_{L^{1,\infty}(\gamma_\infty)}\lesssim_{\varrho}  J_\varrho^{1,\infty} (\mathcal H_t \,f:\,t\in\mathcal I).
\end{equation}
As a consequence,  the weak  jump quasi-seminorms $ J_\varrho^{1,\infty} (\mathcal H_t \,f :\,t\in\mathcal I)$ are unbounded for all 
$\varrho\in[1,2)$.
We remark that in \eqref{ineq:Mariusz} one cannot replace $L^{1,\infty}(\gamma_\infty)$ by $L^1(\gamma_\infty)$, as proved in 
\cite[Lemma 2.5, p. 3395]{Mirek5}; 
this lemma is stated with $\varrho=2$, but the proof holds for any $1\le\varrho \le2$.

In order better to understand   the scheme of the proof of Theorem \ref{thm} in the following sections, we need to go  a little deeper into the proof 
of \cite[Theorem 1.1]{CCS8}, concerning the weak type $(1,1)$ of the $\varrho$-th order variation seminorm of  $\left(\mathcal H_t\right)_{t>0}$.
A basic ingredient of our analysis was
 the Mehler kernel $K_t=K_t(x,u)$,  that is, the integral kernel of the semigroup operator $\mathcal H_t$ with respect to  the invariant measure; see  \eqref{mehler}    for an
  explicit expression of $K_t$.
It is well-known that $K_t$  behaves  very differently 
 in the global and in the local region,
 that is,  roughly speaking, away from or close to the diagonal $x=u$
(more  precise definitions  will be given  in Section
\ref{local1}).
It also matters whether the time parameter $t$ is large or small.
  For this reason, 
the proof of \cite[Theorem 1.1]{CCS8} requires several distinctions both in space and in time.

Some parts of the operator 
$ f \mapsto \| \mathcal H_t f\|_{v(\varrho),\Bbb R_+}$
turn out to be
of weak type $(1,1)$ not only for $\varrho>2$,  but even  %%with respect to the measure $\gamma_\infty$
for $\varrho \ge 1$.  Because of  \eqref{weak_domination_OU}, the corresponding parts of 
 $J_\varrho^{1,\infty} (\mathcal H_t f:\,t>0)$
 will then be controlled by the $L^1(\gamma_\infty)$ norm of $f$.
 %In the light of \eqref{weak_domination_OU}, $J_\varrho^{1,\infty} (\mathcal H_t f(\cdot):\,t>0)$ is also bounded for $\varrho\ge 1$.
 In this way, we shall see in Section \ref{s:local}
how the proof of Theorem \ref{thm} can be reduced to the case of the local region for $t\in (0,1]$.
  %to prove the weak type $(1,1)$ of the jump quasi-seminorms  in the local region, for $t\in (0,1]$.
This case will be treated in the last four sections, where we   compare our semigroup with a convolution semigroup (essentially the heat semigroup), for which jump bounds are known \cite{Liu}.
It is worth mentioning that, differently from \cite{CCS8}, we prove Theorem \ref{thm} without appealing to vector-valued Calder\'on--Zygmund theory.

\medskip

Unexpectedly, the question whether
the jump quasi-seminorm 
of a general Ornstein--Uhlenbeck semigroup  in $\R^n$
defines an operator of strong type $(p,p),\; 1<p<\infty,$ with respect to the invariant  measure, 
is still open. 
Notice however that in the symmetric case (that is, when each operator $\mathcal H_t$ is a self-adjoint operator)
 the boundedness of $J^p_2(\mathcal H_t f:t>0)$  follows
from
\cite[Theorem 1.2]{Mirek1}.

%%In this paper, by exploiting hypercontractivity properties of a general Ornstein--Uhlenbeck semigroup \cite{CM, MF},
%%we are able to prove the boundedness of  $J^p_2(\mathcal H_t f:t\ge 1)$.
%%We also prove   the boundedness of  $J^p_2(\mathcal H_t f:t\in (0,1))$
%%in the local region, so what is missing is the (possible) boundedness of $J^p_2(\mathcal H_t f:t\in (0,1))$ in the global region.

\medskip
 The structure of the article is as follows.  Section~\ref{subadd} gives a subadditivity property for the counting function. Then in  Sections~\ref{s:Gaussian} and \ref{s:Mehler}
we state  basic facts concerning
 the Ornstein--Uhlenbeck semigroup and the Mehler kernel, respectively.
  The distinction between the local and the global part of the Ornstein--Uhlenbeck semigroup operator
 is presented in  Section~\ref{s:local}, and some cases of  Theorem~\ref{thm} are verified (see above).
 The remaining case of Theorem \ref{thm} is proved  in 
Sections~6--9. This is the  local part for $t\in (0,1]$, and   
the localization allows us to deal with one patch of $\R^n$ at a time.
    %from  \ref{s:tsmall_easy} to \ref{s:The main operator} 
 More precisely, in Section~\ref{s:tsmall_easy} 
  we treat the case  when $t\le 1$ is not too small (essentially  $ |x|^{-2} \le t \le 1$),   proving the strong type $(p,p)$ for  $p>1$ and the weak type $(1,1)$
                   %with respect to the invariant (and Lebesgue) measure 
                   for the variation of the local part of $\mathcal H_t  f$, and here  $\varrho \ge 1$.
Dealing with  the local part of $\mathcal H_t f$ for small $t$  (essentially $0<t\le \min(1, {|x|^{-2}}) $) is more involved.
In Section~\ref{s:dec} we decompose 
the semigroup operator into  a sum of three difference operators 
                                  %  $\Delta_t^{(\kappa)}   $, \hskip3pt $\kappa=1,2,3$,
and a main operator.                         % $M_t$.
By means of some technical estimates
one proves                  % that $\Delta_t^{(\kappa)}   $ behave better that $M_t$; in fact,
 in Section~\ref{subsection:diff}
 strong and weak type estimates for the variation of
  the difference operators,  with $\varrho\ge 1$.           %$    \| \Delta_t^{(\kappa)} \,f\|_{v(\varrho), (0, K/ {|x|^2} ]}$,  \hskip3pt  $\kappa=1,2,3$, with $\varrho\ge 1$.
In Section~\ref{s:The main operator} we conclude the proof of  Theorem~\ref{thm} by showing that the main operator 
                                              %$M_t$
satisfies both strong and weak jump inequalities for $\varrho=2$.

\smallskip

It is a pleasure to thank Professor Mariusz Mirek for many helpful and constructive discussions
on the subject of this paper.

\subsection*{Notation}
We will write $C < \infty$ and $c > 0$ for various constants, all of which depend only
on $n$, $Q$ and $B$, unless otherwise explicitly stated.
{\color{black}{The constants involved when we write $\lesssim$ or $\gtrsim$ or $\simeq$ are of this type, except when there is a subscript indicating dependence also on an additional parameter.}}
 The Lebesgue measure of a measurable set $E \subset \R^n$ will be denoted by $|E|$. 
 We write $B(x_0,r)$ for the closed ball of center $x_0$ and radius $r$, and
 $CB(x_0,r)$ for its concentric scaling. By $\mathbb{N}$ we mean $\{0, 1,\dots\}$.

\bigskip

\section{Subadditivity}\label{subadd}

In this section, we describe a subadditivity property of the counting function, in a generic setting.
  For a more comprehensive treatment of the variational and jump inequalities, we refer the reader to \cite{Mirek4, 
Mirek5} or  the recent monograph \cite{Krause}.
\smallskip

Let $\varphi=\varphi(t) $ be a continuous
 real- or complex-valued function   defined in an interval $\mathcal I {\color{black}{\subseteq}} \R_+$.
The $\lambda$-jump counting function of $\varphi$ in $\mathcal I$ may be defined as in \eqref{def:jump}, with $\varphi (t)$ replacing $\mathcal H_t f(x)$.  
Then 
$N_\lambda \big( \varphi(t):\,t\in \mathcal I\big)$ is subadditive in $\mathcal I$.
Moreover, 
it satisfies a kind of  quasi-subadditivity property with respect to $\varphi $;
 more specifically, one has 
\begin{equation}\label{pr:subadd}
 N_\lambda \big(\varphi(t)+\psi(t):\,t\in \mathcal I\big)\le 2\Big(
  N_{\lambda/4} \big(\varphi(t):\,t\in \mathcal I\big)+ N_{\lambda/4} \big(\psi(t):\,t\in \mathcal I\big)\Big)    
  \end{equation}
 for all $\varphi, \psi$ defined on $\mathcal I$ and all $\lambda > 0$.
 For a proof  we refer to  \cite[formulas (2) and (3),  p. 6712]{Jones2}, where a slightly different $\lambda$-jump counting function is also involved.
 This justifies calling $J_\varrho^p(\cdot)$  and $J_\varrho^{1,\infty}(\cdot)$  quasi-seminorms.

  \begin{comment}
  
         Next we will describe a simple way of estimating the variational seminorm, in some cases.
         }}
The $\varrho$-th order variational seminorm of $ \varphi$   on  $\mathcal I$ is defined
 for $1 \le \varrho < \infty$  by
\begin{equation}  \label{def:intro_rho2}
                      \|\varphi\|_{v(\varrho), \mathcal I} =                      \|  \big(\varphi (t) \,:\,t\in \mathcal I\big)\|_{v(\varrho), \mathcal  I} = \sup \left( \sum_{i=1}^{N} |\varphi(t_i) - \varphi(t_{i-1})|^\varrho \right)^{1/\varrho},
\end{equation}
where the supremum is taken over all finite, increasing sequences $\left(t_i \right)_0^N$ of points in $\mathcal I$.
Notice that this quantity is decreasing in $\varrho$.
If in addition  $\varphi \in C^1(\mathcal I)$ and $\varphi' \in L^1(\mathcal I)$, then  for    $1 \le \varrho < \infty$
   {\color{black}{     
 \begin{equation}\label{var}
   \|\varphi  \|_{v(\varrho), \mathcal I} \le   \|\varphi  \|_{v(1), \mathcal I} 
   \le \int_{\mathcal I} |\varphi'(t)|\,dt.
 \end{equation}
                %see \cite[Lemma 2.1]{CCS6}.

\end{comment}

\bigskip

\section{A general Gaussian setting}\label{s:Gaussian}

We introduce some basic notions in the setting of  Ornstein--Uhlenbeck analysis  and recall some facts from \cite{CCS2, CCS3,  CCS4}.
\begin{enumerate}
\item As stated in the introduction,
 $Q$ is  a real, symmetric and positive definite $n\times n$ matrix, and
$B$ a real $n\times n$   matrix whose eigenvalues have negative real parts.
\item The covariance  matrices  are defined by
\begin{equation}\label{defQt}
Q_t=\int_0^t e^{sB}Qe^{sB^*}ds, \qquad \text{ $t\in (0,+\infty]$}.
\end{equation}
Both $Q_t$ and $Q_\infty$ are well defined, symmetric  and positive definite.
Since $Q_t = t\,Q(1+\mathcal O(t))$ as $t \to 0$, one also has
\begin{equation}\label{qt1}
  Q_t^{-1} = t^{-1}\,Q^{-1}(1+\mathcal O(t))
   \end{equation}
  and
\begin{equation}\label{qt2}
\det \, Q_t =t^n\, \det  Q\, (1+ \mathcal O(t))
  \end{equation}
  as $t \to 0$. Further,
\begin{equation}\label{qt3}
 |Q_t^{-1}| \lesssim t^{-1} \quad \mathrm{for} \quad  0<t<1.
  \end{equation}

\item 
$R(x)$ denotes 
the quadratic form 
\begin{equation*}
R(x) ={\frac12 \left\langle Q_\infty^{-1}x ,x  \right\rangle}, \qquad\text{$x\in\R^n$.}
\end{equation*}
\item
$D_t$ is a one-parameter semigroup  of matrices, defined by
\begin{equation}\label{def:Dtx}
D_t =
 Q_\infty
 e^{-tB^*} Q_\infty^{-1} , \qquad t\in\R
.
\end{equation}
For $|t|\le 1$ and  $0\neq x \in \R^n$ these matrices satisfy
 \begin{equation}\label{dtx}
  |x-D_t x|\simeq |t|\,|x|,                            
  \end{equation}
\item  $\gamma_t$  denotes the
normalized Gaussian measure in $\R^n$, given by
$$
d\gamma_t (x)=
(2\pi)^{-\frac{n}{2}}
(\text{det} \, Q_t)^{-\frac{1}{2} }
e^{-\frac12 \langle Q_t^{-1}x,x\rangle}
dx  \,,\qquad \text{ $t\in (0,+\infty]$. }$$
In particular, 
\begin{equation}\label{def:inv_meas}
d\gamma_\infty (x)=
(2\pi)^{-\frac{n}{2}}
(\text{det} \, Q_\infty)^{-\frac{1}{2} }
e^{-R(x)}
dx,  
\end{equation}
and this is the unique invariant measure   for the Ornstein--Uhlenbeck semigroup.
%%\item
%%We shall also use the norm on $\R^n$ \[ |x|_Q := | Q_\infty^{-1/2}\,x|,  \qquad x \in \R^n;\]
%%then $R(x)=\frac12|x|_Q^2$ and  $|x|_Q \simeq |x|$.
\end{enumerate}

\bigskip

 \section{The Mehler kernel and its derivative}\label{s:Mehler}
The Ornstein--Uhlenbeck semigroup operator $\mathcal  H_t $ is given, for all $f\in L^1(\gamma_\infty)$ and all $t>0$,
 by
\begin{align}                       \label{def-int-ker}
 \mathcal H_t
f(x) &=
 \int
K_t
(x,u)\,
f(u)\,
 d\gamma_\infty(u)
  \,, \qquad x\in\R^n.
\end{align}
The integral kernel $K_t$  is called the Mehler kernel and is given, for all
 $x,u\in\R^n$ and $t>0$, by
\begin{align}    \label{mehler}                   %\label{defKRt}
K_t (x,u)\!
=\!
\Big(
\frac{\det \, Q_\infty}{\det \, Q_t}
\Big)^{{1}/{2} }
e^{R(x)}
\exp \Big[
{-\frac12
\left\langle \left(
Q_t^{-1}-Q_\infty^{-1}\right) (u-D_t \,x) , u-D_t \,x\right\rangle}\Big],\qquad
\end{align}
 see  \cite[formula (2.6)]{CCS2}.

The following
estimate, holding for $0<t\leq 1$ and all $x,u\in  \R^n$, is proved in \cite[(3.4)]{CCS2}:
\begin{equation}\label{litet}
   \frac{ e^{R( x)}}{t^{n/2}}\exp\left[-C\,\frac{|u-D_t \,x |^2}t\right]
 \,\lesssim\,   K_t(x,u)
\,\lesssim \, \frac{ e^{R( x)}}{t^{n/2}} \exp\left[-c\,\frac{|u-D_t\, x |^2}t\right].
\end{equation}

%%\section{Jumps} NON VALE PER QUESTA, MA PER i JUMP DEFINITI CON DUE SUCCESSIONI si e ti (guardare Jones-seeger-wright)
%%In \cite{Jones4} a triangle like inequality for jumps numbers is proved. From \cite[(2.2)]{Jones4}
%%one  immediately obtains the following inequality:
%%\begin{equation}\label{ineq:triangle}
%% N_\lambda (\varphi+\psi)\le  N_{\lambda/2} (\varphi)+   N_{\lambda/2} (\psi),\end{equation}for all  continuous
%% real functions $\varphi,\psi $,    defined on a subset $\mathcal I$ of $\R$.

We shall also need  the derivative of $K_t(x,u)$
with respect to $t$, which we denote by $\dot K_t(x,u)$.
For all $(x,u) \in \mathbb R^n\times \mathbb R^n$ and
$t>0$, we proved in \cite[Lemma 4.2]{CCS5} that
\begin{align}\label{def:Kt_deriv}
\dot K_t(x,u) = K_t(x,u) \, N_t^{(0)}(x,u),
\end{align}
where the function $N_t^{(0)}$ is given by
\begin{align}\label{R}
\notag
N_t^{(0)} (x,u)&=-\frac12\,{\tr
\big(Q_t^{-1} \, e^{tB}\, Q\, e^{tB^*}\big)}
+\frac12\,
\left| Q^{1/2}\, e^{tB^*}\, Q_t^{-1}\,(u-D_t\, x)
\right|^2
\notag\\
&\qquad\qquad\qquad
-
\left\langle Q_\infty \,B^*\, Q_\infty^{-1}\, D_t\, x\,,\,
(Q_t^{-1}-Q_\infty^{-1})\,(u-D_t\, x)\right\rangle.
\end{align}

\medskip

For future convenience, we describe  in the Ornstein--Uhlenbeck framework
 a simple way of estimating the variation seminorm, which is sometimes applicable.  In the case when $\varrho = 1$, the variation of $\mathcal H_tf(x)$ is a supremum of sums $\sum |\mathcal  H_{t_i}f(x) - \mathcal H_{t_{i-1}}f(x)|$, which can be estimated by the sum 
$\sum \int_{t_{i-1}}^{t_i} |\partial_t\mathcal H_tf(x)|\,dt$. Since the variation seminorm is decreasing in $\varrho$, we  get for each $\varrho \in [1,\infty)$ and any interval $I \subset \R_+$, at least formally, 
     %we {\Blue{rephrase \eqref{var} in the Ornstein--Uhlenbeck framework.}}  For any $\varrho \in [1,\infty)$, any interval $I \subset \R_+$ and all $x\in\R^n$   the variation seminorm of the Ornstein--Uhlenbeck semigroup with respect to $t\in I$  can be bounded as follows:
 \begin{align}  \label{important}
 \|\mathcal H_tf(x)\|_{v(\varrho), I}&\le     \|\mathcal H_tf(x)\|_{v(1), I}\le      %\int_I \left| \frac{\partial}{\partial t} H_tf(x)\right| dt \notag
    \int_I \left|\frac{\partial}{\partial t} \int K_t(x,u) f(u)\,d\gamma_\infty (u)  \right|\,dt  \\\notag
   % =\int_I \left| \int \dot K_t(x,u) f(u) \,d\gamma_\infty (u) \right|\,dt     \notag  \\
&\le \int         \int_I \big|  \dot K_t(x,u)\big| \,dt \, |f(u)|\,  d\gamma_\infty (u),\qquad f\in L^1(\gamma_\infty).
  \end{align}
  This argument was justified and used in  \cite[Subsection 5.1]{CCS5}.    
       %; the essential point is the swap between differentiation and integration, verified in \cite[Lemma 5.3]{CCS5}. 
  Below we will apply it to some other operators and their kernels.

\bigskip

\section{Localization}\label{s:local}
\label{local1}
The distinction between a local and a global part of the operator is a standard tool in the study of the Ornstein--Uhlenbeck semigroup; we refer to \cite{Mu, Sj, Perez}
for more details on this idea. Here we just describe in a  concise way  the localization procedure introduced in 
 \cite[Section 7]{CCS5}.
 
 We choose a sequence $B_j = B(x_j, 1/(1+|x_j|)),\;j=0,1,\dots,$ of closed balls whose interiors are pairwise disjoint, and they form a maximal family with this property. As shown in \cite[Section 7]{CCS5}, the scaled balls $3B_j = B(x_j, 3/(1+|x_j|))$ then cover  $\mathbb R^n$. The sequence starts with $B_0 = B(0, 1)$, which implies that $|x_j|>1$ for $j \ge 1$.
 
 In  \cite[Section 7]{CCS5} we also introduced smooth, nonnegative funtions $r_j,\;j=0,1,\dots,$ such that 
  $\sum_0^\infty r_j = 1$ in  $\mathbb R^n$, and with  $r_j$ supported in $4B_j$.  The functions 
  $\widetilde r_j,\;j=0,1,\dots,$ also have values in $[0,1]$, and  $\widetilde r_j$ is supported in  $6B_j$ and equals 1 in $5B_j$.
 For the gradients, one has
 \begin{equation}\label{nabla2}   |\nabla  r_j(x)| ,\;\;
\big|\nabla \widetilde r_j (x)\big|\lesssim 1+|x|.
\end{equation}
  The balls  $6B_j$ have bounded overlap, see  \cite[Section 7]{CCS5}. Moreover, the density of $\gamma_\infty$ is essentially constant in each $6B_j$, since 
  \begin{equation}\label{density} 
  e^{R(x)} \simeq e^{R(x_j)}, \qquad x \in  6B_j,
\end{equation}
uniformly in $j$.
  
  \medskip 
 
 Then we let
\begin{equation}\label{def:eta}
 \eta(x,u)=\sum_{j=0}^\infty
\widetilde r_j (x) \,r_j (u), \qquad x,\; u \in \R^n,
\end{equation}
and observe that $0 \le \eta(x,u)\le 1.$
The local and global parts of  $\mathcal H_t$ are defined as
 \begin{align}\label{kernelHloc}
    \mathcal H_t^{\mathrm{loc}}  f(x) =   \int_{\R^n} K_t(x,u)\, \eta (x,u)\, f(u)\,  d\gamma_\infty (u), 
   \qquad x \in \R^n,  
    \end{align}
    and
   %The global part of  $\mathcal H_t$ will be $  \mathcal H_t^{\mathrm{glob}} =  \mathcal  H_t -    \mathcal H_t^{\mathrm{loc}}$, that is,
\begin{equation*}                        %\label{def_7sett_nucleo_gl}
     \mathcal H_t^{\mathrm{glob}}  f(x) =  \mathcal  H_t  f(x) -    \mathcal H_t^{\mathrm{loc}} f(x) =   \int_{\R^n} K_t(x,u)\, (1- \eta (x,u))\, f(u)\,
d\gamma_\infty(u), \qquad x \in \R^n.
\end{equation*}

We observe that
 \begin{align}\label{Hloc}
    \mathcal H_t^{\mathrm{loc}}  f(x) =  \sum_{0}^{\infty} \widetilde r_j (x) \, \mathcal H_t (f\,r_j), 
   \qquad x \in \R^n.  
    \end{align}
This sum is locally finite, since the $\widetilde r_j $ have bounded overlap.

 \medskip

Some parts of our jump operator are easy to treat, by means of the following proposition.

 \begin{proposition}
 
 Let  $1\le \varrho <\infty$.
For  $f \in L^1(\gamma_\infty)$
  \begin{equation}\label{op:oper3}
J^{1,\infty}_{\varrho} \left(\mathcal H_t f:t\in [1,\infty)\right)  \lesssim_{\color{black}\varrho} \|f \|_{L^1(\gamma_\infty)}
\end{equation}
and
   \begin{equation}\label{op:oper4}
J^{1,\infty}_{\varrho} \left(\mathcal H_t^{\mathrm{glob}} f:t\in (0,1]\right)   \lesssim_{\color{black}\varrho} \|f \|_{L^1(\gamma_\infty)}.
\end{equation}
                    %% are  of weak type $(1,1)$ with respect to the measure $\gamma_\infty$.
\end{proposition}

\begin{proof}  
From   \cite[Theorems 4.3 and 6.1]{CCS8} we know that for   $1\le \varrho <\infty$ the operators mapping  $f \in L^1(\gamma_\infty)$ to the functions
  \begin{equation}\label{op:oper1}
  \| \mathcal H_t f(x)\|_{v(\varrho), [1,+\infty)}, \quad x \in \mathbb R^n,
\end{equation}
and
   \begin{equation}\label{op:oper2}
  \|{  \mathcal H}_t^{\mathrm{glob}} f(x)\|_{v(\varrho), (0,1]}, \quad x \in \mathbb R^n,
\end{equation}
are  of weak type $(1,1)$ with respect to the measure $\gamma_\infty$. (Notice that our definition of  ${  \mathcal H}_t^{\mathrm{glob}}$  is that of \cite{CCS5} so the maximal operator estimate used in the proof of \cite[Theorem 6.1]{CCS8} is that of \cite[Theorem 10.1]{CCS5}, not \cite[Theorem 6.2]{CCS8}.)
 From these two estimates and  \eqref{weak_domination_OU}, the proposition follows.
\end{proof}

As a consequence, %%%since the quasi-seminorm $ J_\varrho^{1,\infty} $ is decreasing  in $\varrho$,
in order to prove Theorem \ref{thm} it suffices to consider  the local part of the semigroup with $t\in (0,1]$.
Doing so, we deal with each  term in the sum in \eqref{Hloc}   separately, in the following way.

\begin{proposition}   \label{essential}
  Let $j \in \mathbb{N}$. For any  $f \in L^1(\gamma_\infty)$ 
  \begin{equation*}
  J_2^{1,\infty} \left(\widetilde r_j\, \mathcal H_t (f\,r_j):\,0<t\le 1\right)\lesssim \|f\|_{L^1(\gamma_\infty)}
  %x \mapsto  \sup_{\lambda>0}  \lambda \, \big(N_\lambda \big(\mathcal H_t f(x):  t\in\mathcal I \big)\big)^{1/2},
  \end{equation*}
  uniformly in $j$.
 \end{proposition}
 Notice that  
  \eqref{density} allows us to replace  $\gamma_\infty$ here by Lebesgue measure.

 Because of the bounded overlap of the  $\widetilde r_j$, 
this proposition implies that
\begin{equation*}
  J_2^{1,\infty} \left( \mathcal H_t^{\mathrm{loc}} f:\,0<t\le 1\right)\lesssim \|f\|_{L^1(\gamma_\infty)}.
   \end{equation*}
   This and Proposition \ref{op:oper4} will complete the proof of
 Theorem \ref{thm}, in view of the quasi-subadditivity in Section \ref{subadd}.

To prove Proposition \ref{essential}, we will split the $t$ interval  $(0,1]$ for $j \ge 1$ into 
$[ |x_j|^{-2}, 1]$ (treated in Section \ref{s:tsmall_easy}) and 
 $ (0,   \min(1, |x_j|^{-2})$
(in Sections \ref{s:dec}, \ref{subsection:diff} and \ref{s:The main operator}).   %, for a constant $K>0$ which will be chosen at the beginning of Section \ref{s:tsmall_easy}.
For $j=0$ only the second of these subintervals will be nonempty, since $x_0 =0$.

\bigskip

\section{The local case for  $t$ not too small}\label{s:tsmall_easy}

\begin{proposition} \label{case t not small}
 Let  $1\le \varrho <\infty$. 
   For all $j \ge 1$  the operator that maps   $f$ to the function
  \begin{equation*}
 \|\widetilde r_j\, \mathcal H_t  (f\,r_j)\|_{v(\varrho), [ |x_j|^{-2} , 1]}
\end{equation*}
is of strong type $(p,p)$ for all $p \ge 1$ with respect to  Lebesgue measure,
 uniformly in $j$.
\end{proposition}

Because of \eqref{weak_domination_OU},  this implies that part of Proposition \ref{essential} which deals with 
$t \in [|x_j|^{-2} , 1]$.

\begin{proof}
We fix  $j \ge 1$.
From \eqref{litet}                      %and Lemma \eqref{density}
it follows that for  $x \in 6B_j$, any $u \in \R^n$ and  $|x_j|^{-2} \le t \le 1$ 
\begin{align}\label{dotKeps}
K_t(x,u)
&\lesssim
 e^{R(x)}\,t^{-n/2}\, \lesssim \, e^{R(x_j)}\,|x_j|^{n}.
 \end{align}

  Let $N(x,u)$ be the number
of zeros in $ [|x_j|^{-2} , 1]$ of the function $t\mapsto \dot K_t(x,u)$.
 Proposition 9.1 in \cite{CCS5} implies that $N(x,u) \le  \bar{N}$  for some  constant $\bar N $
that   depends  only on $n$ and $B$.
We denote these zeros by $t_1(x,u)< \dots< t_{N(x,u)}(x,u)$,
and set $t_0(x,u) =|x_j|^{-2}$ and
$t_{N(x,u)+1}(x,u) =1$.
It follows from the fundamental theorem of calculus that
\begin{multline} \label{int_dotK}
\int_{|x_j|^{-2}}^1 \left|\dot K_t(x,u)\right| dt
      %&= \sum_{i=0}^{N(x,u)-1} \int_{t_i(x,u)}^{t_{i+1}(x,u)} |\dot K_t(x,u)|\, dt
               %= \sum_{i=0}^{N(x,u)} \left|\int_{t_i(x,u)}^{t_{i+1}(x,u)} \dot K_t(x,u) dt\right| \\
= \sum_{i=0}^{N(x,u)} \left|K_{t_{i+1}(x,u)}(x,u) - K_{t_i(x,u)}(x,u)\right|  \lesssim \bar N \sup_{|x_j|^{-2}\le t\leq 1}K_{t}(x,u).
         %\\       \leq 2 \sum_{i=0}^{N(x,u)+1}K_{t_i(x,u)}(x,u)
   %\leq 2(\bar N+1) \sup_{0<t\leq 1}K_{t}(x,u)
   %\;\lesssim \;\bar N \sup_{0<t\leq 1}K_{t}(x,u).
\end{multline}

   We can safely apply  \eqref{important} to the operator  $f \mapsto \widetilde r_j\,\mathcal H_t \, (f\,r_j)$.
   This was verified for  $\mathcal H_t$ in  \cite[Subsection 5.1]{CCS5}, and the extra factors $\widetilde r_j$ and 
   $ r_j$ do not matter here. Combined with \eqref{int_dotK} and \eqref{dotKeps}, this yields for 
 any $\varrho \ge 1$ and 
 $f \in L^1(\gamma_\infty)$
 %%%we proceed as in Proposition 	10.2 in "Spectral...", getting \footnote{$\bar N$ is like in Proposition 	10.2 in "Spectral...". }
\begin{align}  \label{important1}
 \|\widetilde r_j(x)\,\mathcal H_t \, (f\,r_j)(x)\|_{v(\varrho), [|x_j|^{-2},1]}&\le    
 \widetilde r_j(x)\, \int \, \int_{|x_j|^{-2}}^1 \big|\dot K(x,u)\big|\,dt\, |f(u)|\,r_j(u)\,d\gamma_\infty (u) \\
                          % \,|
 %\int_I \left| \frac{\partial}{\partial t} H_tf(x)\right| dt \notag
  %  \int         \int_{ [K/ |x_j|^2,1]} \big|  \dot K_t(x,u)\big|   %%%\,dt \, |f(u)|\,  d\gamma_\infty (u)\,  \,\mathbf1_{ \widetilde B_j}      %(x)\notag    \\
 & \le \bar N        \int_{4B_j} \,\sup_{t\in[|x_j|^{-2},1]}
| K_t(x,u)|    \, |f(u)|\,  d\gamma_\infty (u) \,\mathbbm 1_{ 6 B_j}(x)\notag    \\
&\lesssim \bar N\,    e^{R(x_j)}\,|x_j|^n
   \,
\|f\|_{L^1(4B_j, \,\gamma_\infty)}  \,\mathbbm 1_{6 B_j}(x)     %\int_I \left| \frac{\partial}{\partial t} H_tf(x)\right| dt \notag
    \\
 %   &\simeq \Big|B\big(x_j,\frac1{1+|x_j|)}\big)\Big|^{-1}\|g\|_{L^1(4B_j, du)}  \,\mathbf 1_{ \widetilde B_j}(x)   \\
   &\lesssim   \big|4 B_j \big|^{-1}\|f\|_{L^1(4B_j, du)}  \,\mathbbm 1_{ 6 B_j}(x).
  \end{align}
 
 % where, in the penultimate step,  we used once again \eqref{Rconst}.
  The last expression is a simple mean value of $|f|$, and the strong type  $(p,p)$ for all $p \ge 1$ is clear.
   \end{proof}

% Thus both the strong type $(p,p)$ and the  weak type $L^1(B_j,\gamma_\infty) \to L^{1,\infty}(B_j',\gamma_\infty)$ immediately  follow from the boundedness properties of the mean operator.}}
%%Set now
%%%  \[M_j g_j(x):=\left( \big|B\big(x_j,\frac1{1+|x_j|)}\big)\big|^{-1} \int_{B\big(x_j,\frac1{1+|x_j|)}\big)}g_j (u)\,du \right) \,\mathbf 1_{B'_j}(x).\]
%%For all $p\ge 1$,
%%\begin{align*}\|\, \|\mathcal H_t g(x)\|_{v(\varrho), [K/ |x_j|^2,1]}\,\|_{L^p (B_j', \gamma_\infty)}&\lesssim\|M_j g_j\|_{L^p (\gamma_\infty)}.   \end{align*}
%% Thus both the strong type $(p,p)$ and the bound $L^1(B_j,\gamma_\infty) \to L^{1,\infty}(CB_j,\gamma_\infty)$ immediately  follow (since the mean value $M_j g_j$ is no larger %%%than the Hardy--Littlewood maximal function of $f$ at $x_j$).
  % as well.
 %%%Since the mean value $M_jf_j$ is no larger than the Hardy--Littlewood maximal function of $f$ at $x_j$,
%%Thus the bound $L^1(B_j,\gamma_\infty) \to L^{1,\infty}(B'_j,\gamma_\infty)$ immediately  follows, and, by means of  H\"older's inequality, 
%%%the strong type $(p,p)$ as well.
%% {\Red{To conclude, it is enough to sum in $j$.}}

\bigskip

  \section{The local case for small $t$   %$0<t\le \min(1,K/|x_j|^2)$
   : the decomposition}\label{s:dec}
   {\color{black}{
     In this and the following sections, we keep  $j \in \mathbb N$  fixed and prove Proposition \ref{essential} in the remaining case
  $0 < t \le \min(1, |x_j|^{-2})$. For these $t$ one has $t \lesssim 1/|x|^2$, as seen by considering small and large $x_j$ separately. 
  
  The operator  $\mathcal H_t^{\mathrm{loc}} $}} will be decomposed into  a sum of three difference operators 
and a main operator. 
More precisely,       % for $t\in (0,1)$,  $x \in B_j'$ (defined in \eqref{def:Bj'}),  and $ f \inL^1(B_j,\gamma_\infty)$, 
   we write
              %  the  Ornstein--Uhlenbeck semigroup operator as
\begin{align}\label{eq:sum_Ht}
 \mathcal H_t^{\mathrm{loc}}  f(x)=\sum_{\kappa=1}^3 \Delta_t^{(\kappa)} f(x)  + M_t f(x),
%%\int_{\R^n}\widetilde K_t^{(3)}(x,u)\, f(u)\,d\gamma_\infty(u)  ,
\end{align}
where
\begin{align}\label{def_Deltat}
\Delta_t^{(\kappa)} f(x)&=\int_{\R^n}\Big( \widetilde  K_t^{(\kappa-1)}(x,u)-  \widetilde K_t^{(\kappa)}(x,u)\Big)\, f(u)\,d\gamma_\infty(u),\qquad \kappa=1,2,3,\qquad\\
M_t f(x)&=
\int_{\R^n}\widetilde K_t^{(3)}(x,u)\, f(u)\,d\gamma_\infty(u),\label{def_Mt}
\end{align}
 and the kernels $ \widetilde  K_t^{(\kappa)}(x,u)$, \hskip3pt  $\kappa=1,2,3$, are given, respectively, by
  \begin{align}
  \widetilde K_t^{(0)}(x,u)&= K_t(x,u)\,\notag\\
  \widetilde K_t^{(1)}(x,u)&=
\Big(\frac{
 \det Q_\infty}{\det Q_t}\Big)^{1/2}\,    e^{R(x)}\, \exp\Big( -\frac12\,
  \langle (Q_t^{-1} -Q_\infty^{-1})(u- x), \,u- x \rangle   \Big) ,\label{def:tildeKt} \qquad\\
  \widetilde K_t^{(2)}(x,u)&=
\Big(\frac{
 \det Q_\infty}{\det Q_t}\Big)^{1/2}\,    e^{R(x)}\, \exp\left( -\frac1{2t}\,
  \big|Q^{-1/2}(u- x)\big|^2   \right), \label{def:tildeKt_two} \\
  \widetilde K_t^{(3)}(x,u)&=
\frac{(\det Q_\infty)^{1/2}}{ { (\det  Q)}^{1/2}\,  t^{n/2}}\,    e^{R(x)}\, \exp\left( -\frac1{2t}\,
  \big|Q^{-1/2}(u- x)\big|^2   \right) .
\label{def:tildeKt_three} 
  \end{align}

 {\color{black}{
  In the next section, we shall see that the three difference operators satisfy good variational inequalities for all $\varrho \ge 1$.
  Thus the critical    part of $\mathcal H_t^{\mathrm{loc}}$    or    $\mathcal H_t$      is                                 %issues are concentrated in  
   $M_t$.  
   One notices  that    $M_t$
 basically coincides, up to some constants and the factor $e^{R(x)}$, 
            with the operator with kernel  $K_t^c(x-u)$ introduced in \cite[formula (8.3)]{CCS8}. 
The results in  \cite[Section 8]{CCS8} therefore imply that  
the variation operator
  \begin{equation*}
    f \mapsto \|M_t f\|_{v(2), \R_+}                             %, \qquad x\in B_j',
  \end{equation*}
  is not of strong nor weak type $(p,p)$ with respect to $\gamma_\infty$,  for any $p \in [1,\infty)$.
  It follows that the same negative result holds  for  $\mathcal H_t^{\mathrm{loc}}$    and    $\mathcal H_t$.

However,
 we shall prove in Section \ref{s:The main operator} that jump 
 quasi-seminorms with $\varrho = 2$ behave well for  $M_t$, in the way stated in Theorem \ref{thm}.

  Before considering $M_t$,  we  study the three difference operators.                                    
                                  % inthe next section, proving  that they  satisfy good variational inequalities for all $\varrho\ge 1$.
}}
\medskip

\section{The local case for small $t$      %$0<t\le \min(1,K/|x_j|^2)$
: the difference operators     %$\Delta_t^{(\kappa)}$
}\label{subsection:diff}
In order to  estimate the variation  of the difference operators,
we shall apply 
\eqref{important} to each   $\Delta_t^{(\kappa)}$.

The necessary  bounds on the $t$ derivatives of the kernels of  $\Delta_t^{(\kappa)}$
will be proved  in the following subsection. 

\subsection{Some auxiliary bounds}\label{ss:auxiliary_bounds}

First of all, we provide expressions for $\partial_t \,\widetilde K_t^{(\kappa)}$,  similar to \eqref{def:Kt_deriv}.

\begin{comment}

\begin{lemma}\label{kernel-derivative}
For all $(x,u) \in \mathbb R^n\times \mathbb R^n$ and
$t>0$, we have
\begin{align*}
\dot K_t(x,u) = K_t(x,u) \, N_t(x,u),
\end{align*}
where the function $N_t$ is given by
\begin{align}\label{R}
\notag
N_t (x,u)&=
-\frac12\,{\tr
\big(Q_t^{-1} \, e^{tB}\, Q\, e^{tB^*}\big)}
+\frac12\,
\left| Q^{1/2}\, e^{tB^*}\, Q_t^{-1}\,(u-D_t\, x)
\right|^2
\\
&\qquad\qquad\qquad
-
\left\langle Q_\infty \,B^*\, Q_\infty^{-1}\, D_t\, x\,,\,
(Q_t^{-1}-Q_\infty^{-1})\,(u-D_t\, x)\right\rangle.
\end{align}
\end{lemma}
\end{comment}

\begin{lemma}\label{derivate-nucleo}
For all $(x,u) \in \mathbb R^n\times \mathbb R^n$, \hskip4pt
$t>0$ and $\kappa=1,2,3$, 
we have
\begin{align}\label{def:der_t_Kt}
\partial_t \, \widetilde K_t^{(\kappa)}(x,u) = \widetilde K_t^{(\kappa)}(x,u) \,   N_t^{(\kappa)}(x,u),
\end{align}
where 
\begin{align}
 N_t^{(1)} (x,u)&=
 -\frac12\,{\tr
\big(Q_t^{-1} \, e^{tB}\, Q\, e^{tB^*}\big)}
+\frac12\,
\left| Q^{1/2}\, e^{tB^*}\, Q_t^{-1}\,(u-\, x)
\right|^2,\label{RR}\\
  N_t^{(2)} (x,u)&=
 -\frac12\,{\tr
\big(Q_t^{-1} \, e^{tB}\, Q\, e^{tB^*}\big)}
+\frac1{2t^2}\,
\left|Q^{-1/2} (u- x)
\right|^2, \label{RRaa}
\\
  N_t^{(3)} (x,u)&=
 -\frac{n}{2t}
+\frac1{2t^2}\,
\left|Q^{-1/2} (u- x)
\right|^2.\label{RRthree}
\end{align}
\end{lemma}
\begin{proof}
 Differentiating \eqref{def:tildeKt}, \eqref{def:tildeKt_two} and \eqref{def:tildeKt_three} with respect to $t$ and  applying 
 \cite[Lemma~ 4.1]{CCS5},
one easily obtains  \eqref{def:der_t_Kt}.
\end{proof}

In the following two lemmata we shall bound  the $t$ derivative of $\widetilde  K_t^{(\kappa-1)}(x,u)-  \widetilde K_t^{(\kappa)}(x,u)$ for $\kappa=1,2,3$, and first write
\begin{align}\label{diff_der}
& \partial_t\left(\widetilde K_t^{(\kappa-1)}(x,u)- \widetilde K_t^{(\kappa)}(x,u)\right)\\
=\,&\widetilde K_t^{(\kappa-1)}(x,u)\,\left( N_t^{(\kappa-1)}- N_t^{(\kappa)}\right)(x,u)
+ N_t^{(\kappa)}(x,u)\left( \widetilde K_t^{(\kappa-1)}  - \widetilde K_t^{(\kappa)} \right)(x,u). \notag 
\end{align}
%%Here \[ \; N_t^{(0)}(x,u):=N_t(x,u).\]
We start by examining  this difference         
for $\kappa=1$.
\begin{lemma}\label{lemma:est_for_diff_op}
For  $x,\:u\in 6B_j$ and
$0 < t \le \min(1,|x_j|^{-2})$,
one has
 \begin{align}            \label{claim_der}
\left| \partial_t\left( K_t(x,u)- \widetilde K_t^{(1)}
(x,u)\right) \right|&\lesssim (1+|x|)\, \frac {e^{R(x)}} { t^{(n+1)/2} }\,\exp\left(-c\,\frac{|u-x|^2}t\right).
%%\big| \partial_t\Big( \widetilde K_t^{(j)}(x,u)- \widetilde K_t^{(j+1)}(x,u)\Big) \big|&\lesssim e^{R(x)}\,  t^{-n/2 }\,\exp\big(-c\frac{|u-x|^2}t\big), \quad j=1,2.
 %\label{claim_der_two}
\end{align}
\end{lemma}

\begin{proof}
                      We first  bound $N_t^{(0)}(x,u)- N_t^{(1)}(x,u)$, writing
\begin{align*}
 \Big|N_t^{(0)}(x,u)&- N_t^{(1)}(x,u)\Big|=
 \Big|
 \frac12\,
\left| Q^{1/2}\, e^{tB^*}\, Q_t^{-1}\,(u-D_t\, x)
\right|^2
\\
 &
-
\left\langle Q_\infty \,B^*\, Q_\infty^{-1}\, D_t\, x\,,\,
(Q_t^{-1}-Q_\infty^{-1})\,(u-D_t\, x)\right\rangle-\frac12\,
\left| Q^{1/2}\, e^{tB^*}\, Q_t^{-1}\,(u-\, x)
\right|^2\Big|\\
 &\lesssim
 \left|\left\langle Q_\infty \,B^*\, Q_\infty^{-1}\, D_t\, x\,,\,
(Q_t^{-1}-Q_\infty^{-1})\,(u-D_t\, x)\right\rangle
  \right|
  \\&\qquad\quad
  +\Big|\frac12
  \left| Q^{1/2}\, e^{tB^*}\, Q_t^{-1}\,(u- D_t\, x)
\right|^2-\frac12\left| Q^{1/2}\, e^{tB^*}\, Q_t^{-1}\,(u-\, x)
\right|^2
  \Big|
 =:A+B.
 \end{align*}
 Using \eqref{qt3} and then \eqref{dtx}, we obtain
\begin{align*}
 A &\lesssim \frac{|x|}t \,|u-D_t \,x|\lesssim \frac{|x|}t \left(|u- x|+|x-D_t \,x|\right)  \lesssim
 \frac{|x|}t  \, |u- x|+ |x|^2.
\end{align*}
In a similar way,
\begin{align*}
B &=\Big|\frac12\left| Q^{1/2}\, e^{tB^*}\, Q_t^{-1}\,(u- x+x-D_t\, x)
\right|^2-\frac12\left| Q^{1/2}\, e^{tB^*}\, Q_t^{-1}\,(u-\, x)
\right|^2
  \Big|\\
 &=\Big|\frac12\left| Q^{1/2}\, e^{tB^*}\, Q_t^{-1}\,(x-D_t\, x)
\right|^2+\langle
Q^{1/2}\, e^{tB^*}\, Q_t^{-1}\,(u- x), Q^{1/2}\, e^{tB^*}\, Q_t^{-1}\,(x-D_t\, x)
\rangle
  \Big|\\
  &\lesssim |x|^2+\frac1{t^2}\,|u-x| \,t|x|\lesssim
  |x|^2+\frac{|x|}{t}\,|u-x|,
\end{align*}
and thus 
\begin{align*}
 \left|N_t^{(0)}(x,u)- N_t^{(1)}(x,u)\right|\lesssim
 |x|^2+\frac{|x|}{t}\,|u-x|.
\end{align*}
Hence, we have by \eqref{litet}
\begin{align} \label{KNN}
&\left| K_t(x,u) \,\left( N_t^{(0)}(x,u)- N_t^{(1)}(x,u)\right)\right|  \notag\\
&
\lesssim \, \frac{ e^{R( x)}}{t^{n/2}} \exp\left(-c\,\frac{|u-D_t\, x |^2}t\right)\left( |x|^2+\frac{|x|}{t}\,|u-x|\right)     \notag\\
%&\lesssim \, \frac{ e^{R( x)}}{t^{n/2}} \exp\left(-c\,\frac{|u- x |^2}t\right)\, \exp\left(c\,\frac{t^2\, |x |^2}t\right)
%\left( |x|^2+\frac1{t}|u-x| \,|x|\right)\\
     %&\lesssim \, \frac{ e^{R( x)}}{t^{n/2}} \exp\left(-c\,\frac{|u- x |^2}t\right)\,\exp\left(-c\,{t\, |x %    |^2}\right)\,\exp\left(c\,t\,\frac{|u-x| \,|x|}t\right)
   %\,\left( |x|^2+\frac{|x|}{t}\,|u-x|\right)\\
&\lesssim \, \frac{ e^{R( x)}}{t^{n/2}} \,\exp\left(-c\,\frac{|u- x |^2}t\right)\,\exp(C\,{|u-x| \,|x|})
\,\left( |x|^2+\frac{|x|}{t}\,|u-x|\right),
\end{align}
the last step since
\begin{align*}|u - D_t\,x|^2 & =  |u-x|^2 + 2 (u-x)\cdot (x -D _t\,x)+  |x -D _t\,x|^2 \\
& \ge  |u-x|^2 - 2C|u-x|\,t\, |x|.
\end{align*}

To estimate the last, polynomial factor in \eqref{KNN}, we can replace $|u-x|$ by $\sqrt{t}$, if we also reduce the constant $c$ in the same formula. Further, $|x|^2 \le |x|/\sqrt{t}$ by our assumption on $t$, and so the polynomial factor is controlled by $|x|/\sqrt{t}$. The points $x$ and $u$ both have distance at most $6/(1+|x_j|)$ from $x_j$, which implies 
 \begin{align} \label{product}
 |u-x|\,|x| \lesssim (1+|x_j|)^{-1}\,(1+|x_j|) = 1.
  \end{align}
 Thus the last exponential factor in \eqref{KNN} is harmless, and we conclude that
 \begin{align}\label{est_1_der}
\left| K_t(x,u) \,\big( N_t^{(0)}(x,u)-N_t^{(1)}(x,u)\big)\right|
\lesssim
  \frac{ e^{R( x)}}{t^{n/2}}\, \exp\left(-c\,\frac{|u- x |^2}t\right) \,\frac{|x|}{\sqrt t}.
 \end{align}

 \medskip
 
                                                %Continuing the proof of  \eqref{claim_der}, we
Next we consider the term $N_t^{(1)}(x,u)\big( K_t(x,u) - \widetilde K_t^{(1)}(x,u) \big)$,
and  first observe that in view of \eqref{qt3}
  \begin{align}\label{est_Ntildet}
 \big| N_t^{(1)}(x,u)\big|
\lesssim \frac1t+\frac{|x-u|^2}{t^2 }.
\end{align}

We rewrite the long exponent in the expression  \eqref{mehler} for $K_t(x,u)$ by replacing $u-D_t x$ by $u-x+x-D_t x$ and expanding the scalar product. This leads to
\begin{align} \label{KK}
~\\
-\frac12 \left\langle \left(
Q_t^{-1}-Q_\infty^{-1}\right) (u-D_t x) , u-D_t x\right\rangle = 
-\frac12 \left\langle \left(
Q_t^{-1}-Q_\infty^{-1}\right) (u- x) , u- x\right\rangle +E,\\
~
\end{align}
where 
 \begin{align*}
E &=
{-\left\langle \left(
Q_t^{-1}-Q_\infty^{-1}\right) (u-x) , x-D_t x\right\rangle}
-\frac12  \left\langle \big(Q_t^{-1} -Q_\infty^{-1}\big)(x-D_t x), \, x-D_t x \right\rangle.
\end{align*}
From this and the expression \eqref{def:tildeKt} for $\widetilde K_t^{(1)}$,  we see that we can factor out from  $K_t(x, u) - \widetilde K_t^{(1)} (x, u)$   the exponential  of the first right-hand term in \eqref{KK}, and get
%%Rewriting  the difference $K_t(x,u)-\widetilde K_t^{(1)}(x,u)$ by means of  \eqref{mehler},  one can factor out the exponential of the first right-hand term in \eqref{KK} 
\begin{align} \label{Kdiff}
& K_t(x,u)-\widetilde K_t^{(1)}(x,u)\\&=
\left(
\frac{\det \, Q_\infty}{\det \, Q_t}
\right)^{{1}/{2} }
e^{R(x)}\notag   
 \exp\Big( -\frac12
  \left\langle (Q_t^{-1} -Q_\infty^{-1})(u- x), \,u- x \right\rangle   \Big)\,(\exp E-1).
  \end{align}
We  estimate $E$ by means of \eqref{dtx},  \eqref{product} and the bound 
$t \lesssim 1/|x|^2$, obtaining
\begin{align*}
 |E|\lesssim\ |x|\,|u-x|+t|x|^2\lesssim 1.
\end{align*}
Thus 
\begin{align*}
 |\exp E -1|\lesssim |E|\lesssim\ |x|\,|u-x|+t|x|^2,
\end{align*}
and 
\begin{align}                  \label{K-K}
 \left|K_t(x,u)-\widetilde K_t^{(1)}(x,u)\right| \lesssim  \frac{ e^{R( x)}}{t^{n/2}} \, 
  \exp\left( -c \, \frac{ |u- x|^2}{t}   \right) \,\left(|x|\,|u-x|+t|x|^2\right).
  \end{align}
  The polynomial factor here is $t$ times the one in \eqref{KNN}, and arguing as there, we can replace it by $|x|\sqrt{t}$.

Combined with 
\eqref{est_Ntildet}, this yields
\begin{align*}
 \Big| N_t^{(1)}(x,u)\Big(K_t(x,u)-\widetilde K_t^{(1)}(x,u)\Big)\Big|\notag
& \lesssim 
 \frac{ e^{R( x)}}{t^{n/2}} \,
 \exp\left( -c\,
 \frac{ |u- x|^2}{t}   \right)
\,    |x|\sqrt{t} \,                                 %\left(|x|\,|u-x|+t|x|^2\right)\,
 \left( \frac1t+\frac{|x-u|^2}{t^2 }
\right)
\notag\\
& \lesssim 
 \frac{ e^{R( x)}}{t^{n/2}} \,
 \exp\left( -c\,
 \frac{ |u- x|^2}{t} \,  \right) \,\frac{ |x|}{\sqrt{t}}.   
 \end{align*}
Here we reduced the constant $c$ in order to replace $|u-x|$ by $\sqrt t$ in the polynomial factor.

This estimate 
and  \eqref{est_1_der} imply  Lemma \ref{lemma:est_for_diff_op}.                    %\eqref{claim_der}.
\end{proof}

In the following lemma it is shown that the remaining two differences
     %$\widetilde  K_t^{(\kappa-1)}(x,u)-  \widetilde K_t^{(\kappa)}(x,u)$, \hskip3pt $\kappa=2,3$,   
 satisfy better bounds.

\begin{lemma}\label{lemma:est_for_diff_op_two}
Let $x,\; u\in 6B_j$. 
Then for all 
$0 < t \le \min(1, |x_j|^{-2}) $
one has
 \begin{align*}
 \Big| \partial_t\Big( \widetilde K_t^{(\kappa-1)}(x,u)- \widetilde K_t^{(\kappa)}(x,u)\Big) \Big|&\lesssim 
 \frac{ e^{R( x)}}{t^{n/2}} \,\exp\left(-c\,\frac{|u-x|^2}t\right), \quad \kappa=2,3.\quad
                    % \label{claim_der_two}
\end{align*}
\end{lemma}

\begin{proof}
                                                          % first prove \eqref{claim_der_two} for
We start with the case $\kappa=2$, using \eqref{diff_der}.
  By means of \eqref{qt1} and the fact that $e^{tB^*} = I_n(1+\mathcal O(t))$ as $t\to 0$,            
      % \begin{align*}   Q_t^{-1}&=Q^{-1}t^{-1}(1+\mathcal O(t)) \qquad \mathrm{and} \qquad e^{tB^*} = I_n(1+\mathcal O(t)) ,\qquad t\to 0,  \end{align*}
one finds that 
\begin{align*}
    \left| N_t^{(1)}(x,u)- N_t^{(2)}(x,u)\right| &=\frac12 \, \Big| 
\big| Q^{1/2}\, e^{tB^*}\, Q_t^{-1}\,(u-\, x)
\big|^2-\frac1{t^2}\,
\left|Q^{-1/2} (u- x)
\right|^2 \Big|  \\
 &\lesssim\frac1t\,|u-x|^2.
  \end{align*}

Thus by \eqref{def:tildeKt}  we have
\begin{align}\label{other_diff}
\left| \widetilde K_t^{(1)}(x,u) \,\big( N_t^{(1)}(x,u)- N_t^{(2)}(x,u)\big)\right| 
\lesssim \, \frac{ e^{R( x)}}{t^{n/2}} \,\exp\left(-c\,\frac{|u- x |^2}t\right),
\end{align}
for some $c>0$.
 \medskip
To consider the term $ N_t^{(2)}(x,u)\big( \widetilde K_t^{(1)}(x,u) - \widetilde K_t^{(2)}(x,u) \big)$,
                                                      % To prove \eqref{claim_der_two}, 
we first notice that
  \begin{align}\label{est_Ntildet_two}
   \left|  N_t^{(2)}(x,u)\right|
      \lesssim \frac1 t + \frac{|x-u|^2}{t^2 }.
\end{align}
  Then we observe that \eqref{qt1} leads to
 \begin{equation*}
  -\frac12
  \langle (Q_t^{-1} -Q_\infty^{-1})(u- x), \,u- x \rangle = -\frac1{2t} \, |Q^{-1/2}(u- x)|^2  + \mathcal O(|u- x|^2),
  \qquad t\to 0.
 \end{equation*} 
This is analogous to \eqref{KK}, and arguing as in the steps leading from \eqref{KK} to \eqref{K-K}, we obtain
since $|u- x| \lesssim 1$
\begin{align*}
& \left|
\widetilde K_t^{(1)}(x,u) - \widetilde K_t^{(2)}(x,u) \right|
 \lesssim \frac{ e^{R(x)}}{t^{n/2}}
 \, \exp\left( -c\,
 \frac{ |u- x|^2}{t}   \right)\, |u- x|^2.
  \end{align*}
  
Combined with 
\eqref{est_Ntildet_two}, this yields
\begin{align*}               %\label{est_2_der_new}
 \Big| 
  N_t^{(2)}(x,u)&\left( \widetilde K_t^{(1)}(x,u) - \widetilde K_t^{(2)}(x,u) \right)
 \Big|  \notag \\
& \lesssim \frac{ e^{R(x)}}{t^{n/2}}
 \,
{\cancel{{e^{R(x)}}}}  \exp\left( -c\,
 \frac{ |u- x|^2}{t}   \right)\,  
 \left( \frac{|u- x|^2}t+\frac{|u-x|^4}{t^2 }
\right)\, 
%% \Big(|x|\,\sqrt{t}+\frac{t}{|u-x|^2} |u-x|^2\,|x|^2\Big) \\
\notag \\
& \lesssim 
\frac{ e^{R(x)}}{t^{n/2}}
 \, \exp\left( -c\,
 \frac{ |u- x|^2}{t}   \right),  
 \end{align*}
 with a modified value of $c$.

                                                 %%estimate \eqref{claim_der_two}
   The estimate of the lemma for       $\kappa=2$ follows from this and \eqref{other_diff}.
      %and  \eqref{est_2_der_new}.
\medskip

Moving to $\kappa=3$, we observe by \eqref{qt1} that
\begin{align*}
 \left|N_t^{(2)}(x,u)-N_t^{(3)}(x,u)\right|
& =
 \Big|  -\frac12\,{\tr
\big(Q_t^{-1} \, e^{tB}\, Q\, e^{tB^*}\big)}+
\frac{n}{2t} \Big|
\\
&=
 \left|  -\frac1{2t}\,{\tr
\big( \, I_n+\mathcal O(t) \big)}+\frac{n}{2t}\right|
=\mathcal O(1),
 \end{align*}
and \eqref{def:tildeKt_two} implies
\begin{align}\label{other_diff_three}
\left|\widetilde K_t^{(2)}(x,u) \left( N_t^{(2)}(x,u)-N_t^{(3)}(x,u)\right)\right|
\lesssim  \frac{e^{R( x)}}{t^{n/2}} \exp\left(-c\,\frac{|u- x |^2}t\right).
\end{align}
Then we consider the term 
$N_t^{(3)}(x,u)\left( \widetilde K_t^{(2)}(x,u) - \widetilde K_t^{(3)}(x,u) \right)$
and observe that
  \begin{align}\label{est_Ntildet_three}
 \left| N_t^{(3)}(x,u)\right|
\lesssim \frac1t+\frac{|x-u|^2}{t^2 }.
\end{align}
From \eqref{qt2} we see that
          %  Since $\det \, Q_t =t^n\, \det  Q\, (1+ \mathcal O(t))$ as $t \to 0$, we have
 \begin{align*}
& \left|
\widetilde K_t^{(2)}(x,u) - \widetilde K_t^{(3)}(x,u) \right|
\\
&= \big({
 \det Q_\infty}\big)^{1/2}    e^{R(x)}
   \exp\left( -\frac1{2t}\,
  \left|Q^{-1/2}(u- x)\right|^2   \right)
  \left| \frac1{{(\det Q_t)}^{1/2}}-\frac1{ (\det  Q)^{1/2}\,t^{n/2}}\right|
 \\
& \lesssim   \frac{e^{R(x)}} {t^{n/2}}\,  
    \exp\Big( -\frac1{2t}\,   |u- x|^2     \Big)\,t.
    \end{align*}

   With 
\eqref{est_Ntildet_three} we then have
\begin{align*}           %\label{est_2_der_three}
 \Big| 
 N_t^{(3)}(x,u)\,&\Big( \widetilde K_t^{(2)}(x,u) - \widetilde K_t^{(3)}(x,u) \Big)
 \Big|\notag
  \\
& \lesssim 
  \frac{e^{R(x)}}{\,t^{n/2}}\, 
   \exp\left( -\frac1{2t}\,
  |u- x|^2   \right)\, 
 \left( 1+\frac{|x-u|^2}{t }
\right)\, 
%% \Big(|x|\,\sqrt{t}+\frac{t}{|u-x|^2} |u-x|^2\,|x|^2\Big) \\
\notag\\
& \lesssim 
\frac{ e^{R(x)}}{t^{n/2}}
 \, \exp\left( -c\,
 \frac{ |u- x|^2}{t}   \right),  
 \end{align*}
Together with  \eqref{other_diff_three}, this proves  the lemma  for $\kappa=3$.
  \end{proof}

\subsection{Boundedness of the variation operator for the differences}\label{ss:Delta_main_result}

By means of the estimates proved in the previous subsection,
we can finally estimate the variation of  $ \Delta_t^{(\kappa)} \,f$, for $\kappa=1,2,3$.
%%We shall need the following result, proved in 
%%% \cite{CCS5} (see Lemma 8.3 therein).
%%%\begin{lemma}
%%%\label{lemma-integral_dt}
%%%Let $p,\: r\ge 0$ with $p+ r/2 > 1 $. Assume that $\eta(x,u)> 0$  and  $x\neq u$. Then for $\delta>0$
%%%\begin{equation} \label{propclaim}
%%%\int_0^{1} t^{-p}   \exp\left(-\delta\,\frac{|u- D_t \,x |^2}t\right) |x|^{r} \,  dt\le C\,{|u-x|^{-2p-r+2}}.\end{equation}
%%Here the constant $C$ may depend on  $\delta,\: p$ and   $r$, in addition to   $n$, $Q$ and  $B$.
%%\end{lemma}

%Following the approach in  \cite[Section 7]{CCS8}, we aim at  proving the following result, which implies analogous estimates for $ \Delta_t^{(2)} $ and $ \Delta_t^{(3)} $.

  \begin{proposition} \label{pr:difference}
 Let  $1\le \varrho <\infty$ and $j \in \mathbb{N}$.   %%%For all $j\in\N$, 
 The operator that maps   $f \in L^1(\gamma_\infty)$ to the function
  \begin{equation*}
x \mapsto  \|\widetilde r_j(x) \,\Delta_t^{(\kappa)} \,(f\,r_j)(x)\|_{v(\varrho), (0, \min(1,|x_j|^{-2}) ]}, \qquad \;\kappa=1,2,3,
\end{equation*}
is of strong type $(p,p)$ for all $p \ge 1$ with respect to the invariant (and Lebesgue) measure.
This is uniform in $j$.
\end{proposition}

Because of \eqref{weak_domination_OU} and \eqref{domination}, this proposition implies 
\begin{equation}\label{varjumps1}
    J_2^{1,\infty} \left(\widetilde r_j\, \Delta_t^{(\kappa)} \,(f\,r_j):\,0<t\le \min(1,|x_j|^{-2})\right) \lesssim \|f\|_{L^1(4B_j,\gamma_\infty)}, \qquad \;\kappa=1,2,3,
\end{equation}
and for $1<p<\infty$
\begin{equation}\label{varjumpsp}
    J_2^{p} \left(\widetilde r_j\, \Delta_t^{(\kappa)} \,(f\,r_j):\,0<t\le \min(1,|x_j|^{-2})\right) \lesssim \|f\|_{L^p(4B_j,\gamma_\infty)}, \qquad \;\kappa=1,2,3.
\end{equation}

\begin{proof}[Proof of Proposition \ref{pr:difference}]
We start with the case of  $\Delta_t^{(1)}$.
Let  $f \in L^1(\gamma_\infty)$. 
We want to apply an argument like \eqref{important} to the operator 
 $f \mapsto \widetilde r_j\,\Delta_t^{(1)} \,(f\,r_j)$ and its kernel 
 $$
 \widetilde r_j(x)\,\left(K_t(x,u)- \widetilde K_t^{(1)}(x,u)\right)\,r_j(u).
 $$
What must then be verified is only the swap of differentiation and integration in the last step of  \eqref{important}, that is,
\begin{align} \label{swap}
&\frac{\partial}{\partial t}\, \int \widetilde r_j(x)\,\left(K_t(x,u)- \widetilde K_t^{(1)}(x,u)\right)\,r_j(u)\,f(u)\,d\gamma_\infty(u) \notag \\
= &\:\int\widetilde r_j(x)\,\frac{\partial}{\partial t}\, \left(K_t(x,u)- \widetilde K_t^{(1)}(x,u)\right)\,r_j(u)\,f(u)\,d\gamma_\infty(u) 
\end{align}
for  $0 < t< \min(1,|x_j|^{-2}). $
For this, we choose two points $t_0 < t$  in the interval $\left(0,\, \min(1,|x_j|^{-2})\right) $ and replace $t$ by $\tau$ in the right-hand side of \eqref{swap}, which we then integrate $d\tau$ over $t_0 <\tau < t.$ We get
\begin{align} \label{fubini}
  &  \int_{t_0}^{t}\,\int \,\widetilde r_j(x)\,\frac{\partial}{\partial \tau}\, \left(K_{\tau}(x,u) -\widetilde K_{\tau}^{(1)}(x,u)\right)\,r_j(u)\,f(u)\,d\gamma_\infty(u)\,d\tau\\
    =&\:\int\,\int_{t_0}^{t}\,\widetilde r_j(x)\,\frac{\partial}{\partial \tau}\, \left(K_{\tau}(x,u)- \widetilde K_{\tau}^{(1)}(x,u)\right)\,r_j(u)\,f(u)\,d\tau\,d\gamma_\infty(u),
\end{align}
the equality by means of Fubini's theorem. To see that Fubini can  be applied here, we observe that  Lemma \ref{lemma:est_for_diff_op} implies 
\begin{align*}
   &  \int_{t_0}^{t}\,\int\widetilde r_j(x)\,\left|\frac{\partial}{\partial \tau}\, \left(K_{\tau}(x,u)- \widetilde K_{\tau}^{(1)}(x,u)\right)\right|\,r_j(u)\,|f(u)|\,d\gamma_\infty(u)\,d\tau\\
    \lesssim &\: (1+|x|)\,e^{R(x)}\, \int_{t_0}^{t}\frac{d\tau}{\tau^{(n+1)/2}}\, \int|f(u)|\,d\gamma_\infty(u) < \infty.
\end{align*}
Having thus verified \eqref{fubini}, we can rewrite its  right-hand side by a trivial evaluation of the inner integral, and obtain 
\begin{align} 
  &  \int_{t_0}^{t}\,\int\widetilde r_j(x)\,\frac{\partial}{\partial \tau}\, \left(K_{\tau}(x,u)- \widetilde K_{\tau}^{(1)}(x,u)\right)\,r_j(u)\,f(u)\,d\gamma_\infty(u)\,d\tau\\
    =&\:\int\widetilde r_j(x)\,  \big(K_t(x,u)- \widetilde K_t^{(1)}(x,u)\big)\,r_j(u)\,f(u)\,d\gamma_\infty(u) \\
    -&\:\int\widetilde r_j(x)\,   \big(K_{t_0}(x,u)- \widetilde K_{t_0}^{(1)}(x,u)\big)\,r_j(u)\,f(u)\,d\gamma_\infty(u).
\end{align}
Differentiating this equation with respect to $t$, we obtain \eqref{swap} as desired.

%%{\Yellow{
%%  For this, one can follow the arguments in  \cite[Subsection~5.1]{CCS5}, in particular
  %%\cite[Lemma~5.3]{CCS5}\footnote{\Red{Peter, non sarebbe meglio qui scrivere un po' di più sull'argomento evitando il riferimento a  \cite[Lemma~5.3]{CCS5}?
%%  Il motivo è che i referee spesso hanno scritto che i nostri lavori non sono consistenti perché facciamo troppi riferimenti a lavori passati.}}; what one needs in our case is the %%estimate
%%  \begin{equation*}
%%  \int_{t_0}^{t} \sup_{u \in \mathbb{R}^n}  \left| \frac{\partial}{\partial \tau} \Big(K_\tau(x,u)- \widetilde K_\tau^{(1)}(x,u)\Big)\right|\,d\tau < \infty,
%%  \end{equation*}
 %% where $x \in \mathbb R^n$ and  $0 < t_0 \le t < \min(1,|x_j|^{-2})$. Here we estimated the  factors $\widetilde r_j$ and  $ r_j$ by 1.
%%  This follows immediately from Lemma \ref{lemma:est_for_diff_op}.}}

We can thus apply  \eqref{important} with the indicated kernel. Together with Lemma \ref{lemma:est_for_diff_op}, this  yields
\begin{align} \label{4.4}
&\|\widetilde r_j(x) \,\Delta_t^{(1)} (f\,r_j)(x)\|_{v(\varrho), (0, \min(1,|x_j|^{-2})]}\\
\lesssim &
\int\,\int_{0}^{\min(1,|x_j|^{-2})}\,\widetilde r_j(x)\,\left|\frac{\partial}{\partial t}\, \left(K_{\tau}(x,u)- \widetilde K_{\tau}^{(1)}(x,u)\right)\right|\,r_j(u)\,dt\,|f(u)|\,d\gamma_\infty(u)\\
\lesssim &\,
(1+|x|)\, \,e^{R(x)} \int \int_0^{\min(1,|x_j|^{-2})} t^{-(n+1)/2 }  \,e^{-c\,\frac{|u-x|^2}t} \,dt\, 
\,\,| f(u)|\,r_j(u)\, d\gamma_\infty (u)\, \mathbbm{1}_{6B_j}(x).
\end{align}
In the inner integral here, we make the transformation $s = |u-x|^2/t$ and find
\begin{align}\label{inner}
 &\int_0^{\min(1,|x_j|^{-2})} t^{-(n+1)/2 }   \,\exp\big(-c\,\frac{|u-x|^2}t\big) \,dt \\
= &
\int_{|u-x|^2/\min(1,|x_j|^{-2})}^{\infty}  |u-x|^{1-n }  \, s^{\frac{n-3}{2}} \, e^{-cs}\,ds.
\end{align}  
If $n>1$, the last integral is at most $C\,|u-x|^{1-n}$. 
But if  $n=1$, it is bounded by 
\begin{equation*}    
C + \log_+\frac {\min(1,|x_j|^{-2})} {|u-x|^2},
\end{equation*}
where
$ \log_+\beta=\max( 0, \beta)$  for $\beta>0$.
We insert these estimates in the last double integral in \eqref{4.4}, after observing that there  
  $u-x \in {4B_j} - {6B_j} \subset B(0, 10/(1+|x_j|))$ and also $e^{R(u)}\simeq e^{R(x)} \simeq e^{R(x_j)}$. Thus we can replace
  $d\gamma_\infty (u)$ by $du$ and delete the factor $e^{R(x)}$, in this integral.
Further,   $1+|x| \lesssim 1+|x_j|$ there.  For  $n>1$ we get

\begin{align*}
\|\widetilde r_j(x) \,\Delta_t^{(1)} (f\,r_j)(x)\|_{v(\varrho), (0, \min(1,|x_j|^{-2})]}
\lesssim (1+|x_j|)\, \int \,|u-x|^{1-n}\,| f(u)|\,r_j(u)\,du\, \mathbbm{1}_{6B_j}(x).  
   \end{align*}
The right-hand side here is essentially a convolution, and we can estimate it by
\begin{align*}
 |f\,r_j|* g(x),              %\, \mathbbm{1}_{6B_j}(x), 
 \end{align*} 
 where
  $g(y)= C\,(1+|x_j|)\,|y|^{1-n }\,\mathbbm{ 1}_{B(0, 10/(1+|x_j|))}(y)$, when $n>1.$
  
For $n=1$ we arrive at a similar convolution, but now
\begin{equation*}    
g(y) = (1+|x_j|)\, \left(C+2\log_+\frac {\min(1,|x_j|^{-1})} {|y|}\right)\,\mathbbm{ 1}_{B(0, 10/(1+|x_j|))}(y).
\end{equation*}

  In both cases, we find that 
  $\|g \|_{L^1(dy)} \lesssim 1$, and so convolution by $g$ defines a bounded operator on all $L^p(dx)$ spaces.  
  
This proves Proposition \ref{pr:difference} for  $\kappa = 1$.

 The cases  $\kappa = 2,\;3$ can be done in exactly the same way, since, as already observed,  Lemma \ref{lemma:est_for_diff_op_two} gives a stronger estimate than Lemma \ref{lemma:est_for_diff_op}.

  The proposition is proved. 
 \end{proof}

 %%%OLD VERSION
 \begin{comment}

\[\Psi_t (u):=t^{-n/2}\,\exp\big(-c\frac{|u|^2}t\big)\]
is, up to some constant, an approximation to the identity in $\mathbb R^n$. By \eqref{rel:exponential}
one has $   e^{R(x)} \lesssim e^{R(x_j)} $, whence
\begin{align}\label{try_12february}
& \norm{
 \,
\int
 \partial_t   \Big( \widetilde  K_t^{(0)}(x,u)-  \widetilde K_t^{(1)}(x,u)\Big)\, g(u)\,du}_{L^2(B_j', dx)}
\le t^{-1/2 }\,  |x_j|\, \,e^{R(x_j)}  \norm{
\Psi_t *\, |g|}_{L^2(B_j', dx)}\notag\\
&\qquad = t^{-1/2 }\,  |x_j|\, \,e^{R(x_j)}  \norm{
\Psi_t *\, (|g|\mathbf{1}_{B_j})}_{L^2(B_j', dx)}\notag\\
&\qquad \le  t^{-1/2 }\,  |x_j|\, \,e^{R(x_j)}  \norm{
\Psi_t *\, (|g|\mathbf{1}_{B_j})}_{L^2(\R^n, dx)}\notag\\
&\qquad \le t^{-1/2 }\,  |x_j|\, \,e^{R(x_j)}  \norm{
\Psi_t}_{ L^1(\R^n,dx)}  \||g|\mathbf{1}_{B_j}\|_{L^2(\R^n, dx)}\notag\\
&\qquad \le t^{-1/2 }\,  |x_j|\, \,e^{R(x_j)}  \norm{
\Psi_t}_{ L^1(\R^n,dx)}  \|g\|_{L^2(B_j, dx)}\notag\\
&\qquad \lesssim t^{-1/2 }\,  |x_j|\, \,e^{R(x_j)}  
 \|g\|_{L^2(B_j, dx)}.\end{align}  
\end{comment}

{\Red{

}}

\bigskip

    \section{The local case for small $t$     %$0<t\le  \min(1,K/|x_j|^2)$
    : the main operator    $M_t$
    }\label{s:The main operator}

We are left with the estimate of the main part  $M_tf$ in \eqref{eq:sum_Ht}, that is, the integral operator with the kernel 
 $ \widetilde  K_t^{(3)}$ defined in \eqref{def:tildeKt_three}.
 
\begin{proposition}\label{prop:Mt}
Let $j \in \mathbb N$. For all $p>1$ one has 
\begin{equation}\label{stima_pp_Mt} 
J_2^p
\left(\widetilde r_j\, M_{t}(f\,r_j):0<t\le \min(1,|x_j|^{-2})
\right)
\lesssim_{\color{black}p}\|f\,r_j\|_{L^p(\gamma_\infty)}
\end{equation} 
 and, moreover, 
\begin{equation}\label{stima_1inf_Mt} 
J_2^{1,\infty}
\left(\widetilde r_j\, M_{t}(f\,r_j):0<t\le \min(1,|x_j|^{-2})\right)\lesssim\|f\,r_j\|_{L^1(\gamma_\infty)}.
\end{equation}
The implicit constants in  these inequalities are uniform in $j$.
\end{proposition}

\begin{proof}
Let   for $t>0$
\[
\psi_t(y)=\Big(\frac{
 \det Q_\infty}{\det Q}\Big)^{1/2}\,  
{t^{-n/2}}\,\exp\big(-\frac1{2t}\,|Q^{-1/2}\,y|^2\big), \qquad y\in\R^n.
\]
             Then
\begin{align*} 
 M_{t}(f\,r_j)(x)&=   e^{R(x)}
\Big( \psi_t* \big(fr_je^{-R(\cdot)}\big)\Big)(x).
%%&\simeq e^{R(x_j)}\left( \psi_t* \big(fe^{-R(\cdot)}\big)\right)(x),
\end{align*}
 for $f \in L^1(\gamma_\infty)$.  
We recall that  $e^{R(x)}\simeq e^{R(x_j)}$ for  $x \in 6 B_j \supset \mathrm{supp} \, \widetilde r_j $.

If $  \widetilde r_j(x)\, M_{t}(f\,r_j)(x) = \widetilde r_j(x)\,e^{R(x)}
\left( \psi_t* \big(fr_je^{-R(\cdot)}\big)\right)(x)$ has a jump of amplitude larger than $\lambda $ for some $x \in \mathrm{supp} \, \widetilde r_j$,
 then
 $\psi_t* \big(fr_je^{-R(\cdot)}\big)(x)$ has a jump of amplitude larger than $c\,\lambda \, e^{-R(x_j)}$ for some $c>0$.
 As before, the jumps are in the variable $t$.
Thus
\begin{align}\label{intermediate_psi} 
{N_\lambda \big(\widetilde r_j(x)\,M_{t}(f\,r_j)(x)}\big)
\le
N_{c\,\lambda  e^{-R(x_j)}}\left(
 \psi_t* \big(fr_je^{-R(\cdot)}\big)\right)(x)\, \mathbbm 1_{6B_j}(x),
\end{align}
and  for all $\alpha>0$,
\begin{align*}
& \gamma_\infty\left\{ x\in 6 B_j:\,
 \lambda \,\sqrt{ N_\lambda \left(\widetilde  r_j(x)\, M_{t}(f\,r_j)\right)(x)}>\alpha\right\}\\
 &\lesssim
\gamma_\infty\left\{ x\in 6 B_j:\,    \lambda\, \sqrt{N_{c\lambda  e^{-R(x_j)}}\, \left(\psi_t *\big(fr_je^{-R(\cdot)}\big)\right)(x) }>\alpha \right\}
 \\&=
\gamma_\infty\left\{ x\in 6 B_j:\, c\,
 e^{-R(x_j)}\, \lambda \, \sqrt{N_{c\lambda  e^{-R(x_j)}}\, \left(\psi_t *\big(fr_je^{-R(\cdot)}\big)\right)(x)} > c \, e^{-R(x_j)} \, \alpha \,\right\}
 \\&\simeq
 e^{-R(x_j)}
\left|\left\{ x\in 6 B_j:\,  c 
 e^{-R(x_j)} \,\lambda \,\sqrt{N_{c\lambda  e^{-R(x_j)}}\, \left(\psi_t *\big(fr_je^{-R(\cdot)}\big)\right)(x)} > \, c\, e^{-R(x_j)}\, \alpha \right\}\right|.
\end{align*}
                  %where in the last step we  pass from Gaussian to Lebesgue measure and use once again the fact that $e^{-R(x)}\simeq e^{-R(x_j)}$ for $x\in B_j'$.

Now we  apply  Theorem 2.4 in \cite{Liu} to the last expression here, getting 
\begin{align*}
\gamma_\infty & \Big\{ x\in 6 B_j\,:\, 
 \lambda \,\sqrt{ N_\lambda \left(\widetilde  r_j(x)\,M_{t}(f\,r_j)\right)(x)}>\alpha\Big\}  \\
 \lesssim & \,
\frac{ e^{-R(x_j)}}{ e^{-R(x_j)}\,\alpha} \,
\left\|\psi_t *\big(fr_je^{-R(\cdot)}\big)\right\|_{L^1(du)}
   % \lesssim & \, \frac{ e^{-R(x_j)}}{ e^{-R(x_j)}\,\alpha} \,\int_{B_j} |f(u)|\,r_j(u)\,e^{-R(u)}\,du \,
\lesssim \,\frac{1}{\alpha}\,  \|f\,r_j\,e^{-R(\cdot)}\|_{L^1(du)}
 = \frac{1}{\alpha}\, \|f\,r_j\|_{L^1(\gamma_\infty)},
 \end{align*}
where we also used  the fact that  convolution with   $\psi_t $ defines a bounded operator on $L^1(du)$. 
 This is uniform in $\lambda>0$  and in $j$, and proves \eqref{stima_1inf_Mt}.

We now consider the strong type $(p,p)$ for $1<p<\infty$. Starting from \eqref{intermediate_psi},
 moving to Lebesgue measure  and then applying    \cite[Theorem 2.4]{Liu}, one finds 
\begin{align*} 
&\norm{\, \lambda\sqrt{N_\lambda \big(\widetilde  r_j(\cdot)\,M_{t}(f\,r_j)\big)(\cdot)}\,}_{L^p(6 B_j, \gamma_\infty)}
\le \norm{\, \lambda\sqrt{
N_{c\,\lambda  e^{-R(x_j)}}\left(
\big( \psi_t* \big(fr_je^{-R(\cdot)}\big)\right)
}}_{L^p(6 B_j, \gamma_\infty)}
\\
&=
c^{-1}
e^{R(x_j)}\norm{\,c 
 e^{-R(x_j)}  \lambda\sqrt{
N_{c\,\lambda  e^{-R(x_j)}} \left(
\big( \psi_t* \big(fr_je^{-R(\cdot)}\big)\right)}\,}_{L^p(6 B_j, \gamma_\infty)}\\
&\simeq_{\color{black}p} 
e^{R(x_j)}
\, e^{-R(x_j)/p} 
\norm{\,c 
 e^{-R(x_j)}  \lambda\sqrt{
N_{c\,\lambda  e^{-R(x_j)}} \left(
\big( \psi_t* \big(fr_je^{-R(\cdot)}\big)\right)}\,}_{L^p(6 B_j, dx)}\\
%%&\lesssim_p e^{R(x_j)/2}\norm{ f(\cdot)e^{-R(\cdot)})}_{L^p(B_j, \gamma_\infty)}\\
&
\lesssim_p 
e^{R(x_j)}
\, e^{-R(x_j)/p} 
\left(\int\big|
 f(u)\,r_j(u)\,e^{-R(u)}\big|^p du\right)^{1/p}
%%%&\simeq_p  e^{R(x_j)}\, e^{-R(x_j)/p} \Big(\int_{B_j'}\big| f(x)\big|^p \,e^{-R(x)p} \,dx\Big)^{1/p}\\
\simeq_p 
\| f\,r_j \|_{L^p(\gamma_\infty)},
              %\Big( \int| f(u)\,r_j(u)|^p  \,e^{-R(u)}\,dx\Big)^{1/p}
   %\\
%&\simeq_p c^{-1}\,\norm{f}_{L^p(B_j, \gamma_\infty)}
   %\\ &\simeq_p \norm{f}_{L^p(B_j, \gamma_\infty)},
\end{align*}
uniformly in $\lambda>0$ and $j$. This is \eqref{stima_pp_Mt}, and Proposition \ref{prop:Mt} is proved.
 \end{proof}
 
 Now the remaining part of Theorem 5.2 with $t \in (0, \min(1,|x_j|^{-2})]$ follows as a combination of 
 \eqref{varjumps1} and \eqref{stima_1inf_Mt}. Theorem \ref{thm} is completely proved.

\end{document}